\documentclass{amsart}
\usepackage{amsmath}
\usepackage{amsfonts}
\usepackage{amssymb}
\usepackage[utf8]{inputenc}
\usepackage{amssymb}
\usepackage{mathrsfs}
\usepackage{amsthm}
\usepackage{xcolor}
\usepackage{float}
\usepackage{tikz-cd}

\usepackage{enumerate}

\usepackage{graphicx}

\newtheorem{thrm}{Theorem}[section]
\newtheorem{cor}[thrm]{Corollary}
\newtheorem{lemma}[thrm]{Lemma}

\newtheorem{question}[thrm]{Question}

\newtheorem{claimnumber}{Claim}

\theoremstyle{definition}
\newtheorem{defin}[thrm]{Definition}

\newcommand{\NN}{\mathbb{N}}
\newcommand{\RR}{\mathbb{R}}

\begin{document}

\title[Every Lindel\"of scattered subspace of $\Sigma(\mathbb{R}^\Gamma)$ is $\sigma$-compact]{Every Lindel\"of scattered subspace of a $\Sigma$-product of real lines is $\sigma$-compact}
\author{Antonio Avil\'es}
\address{Universidad de Murcia, Departamento de Matem\'{a}ticas, Campus de Espinardo 30100 Murcia, Spain.} \email{avileslo@um.es}
\author{Miko\l aj Krupski}
\address{Universidad de Murcia, Departamento de Matem\'{a}ticas, Campus de Espinardo 30100 Murcia, Spain\\ and \\ Institute of Mathematics\\ University of Warsaw\\ ul. Banacha 2\\
02--097 Warszawa, Poland }
\email{mkrupski@mimuw.edu.pl}

\begin{abstract}
We prove that every Lindel\"of scattered subspace of a $\Sigma$-product of first-countable spaces is $\sigma$-compact. In particular, we obtain
the result stated in the title. This answers some questions of Tkachuk from [Houston J. Math. 48 (2022), no. 1, 171--181].
\end{abstract}



\maketitle

\section{Introduction}

All spaces under consideration are Tychonoff topological spaces.
In the present note we are concerned with Lindel\"of scattered spaces. Recall that a space $X$ is \textit{scattered} if every nonempty subset of $X$ has a relative isolated point. Although Lindel\"of scattered spaces need not be $\sigma$-compact (consider the one point Lindel\"ofication $L(\omega_1)$ of an uncountable discrete set), they are quite close to being $\sigma$-compact. Indeed, it is well known that any Lindel\"of scattered space is a Menger space (see \cite[T.129]{Tk1} and \cite[Theorem 2.7]{Ar}). Actually, one can show that any such space is even Hurewicz\footnote{A space $X$ is \textit{Menger} (resp., \textit{Hurewicz})
if for every sequence $(\mathscr{U}_n)_{n\in \NN}$ of open covers of $X$, there is a sequence $(\mathscr{V}_n)_{n\in\NN}$
such that for every $n$, $\mathscr{V}_n$ is a finite subfamily of $\mathscr{U}_n$ and the family $\bigcup_{n\in\NN} \mathscr{V}_n$ covers $X$
(resp., every point of $X$ is contained in $\bigcup\mathscr{V}_n$ for all but finitely many $n$'s).}; this is a direct consequence of \cite[T.129]{Tk1} and \cite[Theorem 3.2]{K}. One may wonder under what additional conditions a Lindel\"of scattered space is actually $\sigma$-compact. It was shown by Tkachuk in \cite{Tk} that any Lindel\"of scattered subset of a $\sigma$-product of first-countable spaces is $\sigma$-compact. It remained unclear, however, what happens with Lindel\"of scattered subsets of Eberlein compacta. Indeed, the main question that was left open in \cite{Tk} is the following one:

\begin{question}\cite[Question 4.1]{Tk}\label{main question}
Is every Lindel\"of scattered subset of an Eberlein compact space $\sigma$-compact?
\end{question}

Let us recall that a compact space $K$ is \textit{Eberlein} compact
if $K$ is homeomorphic to a weakly compact subset of a Banach space. This is equivalent to saying that, for some $\Gamma$, the space $K$ is homeomorphic to a compact subset of
$\{x\in \RR^\Gamma :\forall \varepsilon>0\;|\{\gamma\in \Gamma:x(\gamma)>\varepsilon\}|<\omega\}$. For a given family of spaces $\{Y_\gamma:\gamma\in \Gamma\}$ and an element $a\in \prod_{\gamma\in \Gamma} Y_\gamma$, the \textit{$\Sigma$-product} $\Sigma(\prod_{\gamma\in \Gamma}Y_\gamma, a)$ is the following subset of the product of the family $\{Y_\gamma:\gamma\in \Gamma\}$:
$$\Sigma(\prod_{\gamma\in \Gamma}Y_\gamma, a)=\Bigl\{y\in \prod_{\gamma\in \Gamma}Y_\gamma:|\{\gamma\in \Gamma:y(\gamma)\neq a(\gamma)\}|\leq\omega\Bigr\}.$$

The aim of this note is to give an affirmative answer to Question \ref{main question}. In fact, we shall prove a much stronger result that settles a number of open questions posed in \cite{Tk}. Our main result reads as follows:

\begin{thrm}\label{main}
 If $X$ is a Lindel\"of scattered subspace of a $\Sigma$-product $\Sigma(\prod_{\gamma\in \Gamma}Y_\gamma, a)$, of first-countable spaces $Y_\gamma$, then $X$ is $\sigma$-compact.
\end{thrm}

In particular we obtain the following:

\begin{cor}\label{main_cor}
 If $X$ is a Lindel\"of scattered subspace of a $\Sigma$-product $\Sigma(\mathbb{R}^\Gamma)$, then $X$ is $\sigma$-compact.
\end{cor}

Here
$\Sigma(\RR^\Gamma)=\Sigma(\RR^\Gamma,\mathbf{0})=\left\{ x\in \RR^\Gamma:|\{\gamma\in \Gamma:x(\gamma)\neq 0\}|\leq \omega \right\}$.
Compact subsets of $\Sigma(\RR^\Gamma)$ are called \textit{Corson compacta}. Clearly, every Eberlein compact space is Corson compact. Thus we have:

\begin{cor}\label{main_cor_1}
If X is a Lindel\"of scattered subspace of a Corson compact
space, then X is $\sigma$-compact. In particular, all Lindel\"of scattered subspaces of
any Eberlein compact space are $\sigma$-compact.
\end{cor}

Theorem \ref{main} solves Question 4.8 from \cite{Tk} and therefore it
automatically provides answers to Questions 4.1--4.6 as well as to Question 4.9 from that paper (cf. Corollaries \ref{main_cor} and \ref{main_cor_1}). Moreover,
as an immediate consequence of Corollary \ref{main_cor}, we give an affirmative answer to Question 4.5 from \cite{LT} (see Corollary \ref{corollary-delta-spaces} below).

The general scheme of our proof of Theorem \ref{main} is essentially the same as in \cite[proof of Theorem 3.2]{Tk}. A new ingredient is Lemma \ref{Alster} below. The lemma itself is new but its proof is very similar to the reasoning of Alster from \cite{A}.

\section{On Lindel\"of scattered subspaces of certain $\Sigma$-products}

For a given space $X$, by $X'$ we denote the set of all accumulation points of $X$. For an ordinal $\alpha$, we define the Cantor-Bendixson $\alpha$-th derivative $X^{(\alpha)}$ of $X$ recursively, as follows:
\begin{itemize}
 \item $X^{(0)}=X$
 \item $X^{(\alpha+1)}=(X^{(\alpha)})'$
 \item $X^{(\alpha)}=\bigcap_{\beta<\alpha} X^{(\beta)}$ if $\alpha$ is a limit ordinal.
\end{itemize}

A space $X$ is scattered if and only if there is an ordinal $\alpha$ such that $X^{(\alpha)}=\emptyset$. For a scattered space $X$, its \textit{Cantor-Bendixson rank} $ht(X)$ is the minimal ordinal number $\alpha$ for which $X^{(\alpha)}=\emptyset$.
Note that if $X$ is compact scattered, then $ht(X)$ is always a successor ordinal.

\begin{defin}
 Suppose that $X$ is a subspace of a space $Z$. Let $\mathscr{U}=\{U_t:t\in T\}$ be an indexed family of subsets of $Z$. We say that $\mathscr{U}$ is \textit{point-finite in $X$} if for every $x\in X$ the set $\{t\in T:x\in U_t\}$ is finite.
\end{defin}

\begin{defin}
Let $X$ be a subspace of a space $Z$.
Given a family $\mathscr{U}$ of subsets of $Z$, we say that $\mathscr{U}$ is \textit{block-point-finite in $X$} if for some set $\Gamma$ one can write $\mathscr{U}=\bigcup\{\mathscr{U}_\gamma:\gamma\in \Gamma\}$ where each family $\mathscr{U}_\gamma$ is countable and the indexed family
$\{\bigcup\mathscr{U}_\gamma:\gamma\in \Gamma\}$
is point-finite in $X$.
\end{defin}

Given an indexed family $\mathscr{A}=\{A_t:t\in T\}$ of subsets of a space $X$, we say that a family $\{U_t:t\in T\}$ is an \textit{open expansion} of $\mathscr{A}$ if for every $t\in T$, the set $U_t$ is open in $X$ and satisfies $A_t\subseteq U_t$. Suppose that $\mathscr{U}$ and $\mathscr{V}$ are two families of subsets of a space $X$. We say that $\mathscr{V}$ is an \textit{open refinement} of $\mathscr{U}$ if $\mathscr{V}$ consists of open subsets of $X$ and for every $V\in \mathscr{V}$ there is $U\in \mathscr{U}$ with $V\subseteq U$.

\medskip

The following fact was noted by Tkachuk  in \cite{Tk}.

\begin{lemma}\cite[Lemma 3.1]{Tk}\label{lemma_Tkachuk}
 Suppose that $Z$ is a space and $\{G_t:t\in T\}$ is a family of $G_\delta$-subsets of $Z$ such that the set $\bigcup\{G_t:t\in T\}$ is disjoint from a set $X\subseteq Z$. If the family $\{G_t:t\in T\}$ has an open expansion $\{V_t:t\in T\}$ which is point-finite in $X$, then there is a $G_\delta$-subset $G$ of $Z$ such that $\bigcup\{G_t:t\in T\}\subseteq G\subseteq Z\setminus X$.
\end{lemma}

Although the next lemma seems to be new, its proof presented below is very similar to the proof of \cite[Proposition]{A} due to Alster (cf. also \cite[Lemma 3.5]{AK}).

\begin{lemma}\label{Alster}
 Let $X$ be a scattered subspace of a space $Z$ and let $\mathscr{U}$ be a family of open subsets of $Z$. Suppose that for every $x\in X$ the family $\{U\in \mathscr{U}:x\in \overline{U}\}$ is countable. Assume further that $\overline{U}\cap X$ is $\sigma$-compact (maybe empty) for all $U\in \mathscr{U}$. Then $\mathscr{U}$ has an open refinement $\mathscr{V}$ such that $\bigcup\mathscr{V}=\bigcup\mathscr{U}$, the intersection $\overline{V}\cap X$ is $\sigma$-compact (maybe empty) for all $V\in \mathscr{V}$ and $\mathscr{V}$ is block-point-finite in $X$.
\end{lemma}
\begin{proof}
 We proceed by induction on the size $|\mathscr{U}|$ of the family $\mathscr{U}$. If $\mathscr{U}$ is countable, there is nothing to do for we can simply take $\mathscr{V}=\mathscr{U}$. Suppose that $|\mathscr{U}|=\kappa>\omega$ and assume that our result holds for families of cardinality less than $\kappa$. Enumerate $\mathscr{U}=\{U_\alpha:\alpha<\kappa\}$.
 For each $U\in \mathscr{U}$ let us fix a countable family $\{K_n(U):n\in \omega\}$ of compact sets (maybe empty) such that
\begin{equation}\label{equation_sigma_compact}
\overline{U}\cap X=\bigcup_{n\in\omega}K_n(U).
\end{equation}

Given any subfamily $\mathscr{W}$ of $\mathscr{U}$, we set
$$\mathscr{W}'=\{K_n(U):U\in \mathscr{W} \mbox { and }n\in\omega\}$$
 For a compact scattered set $F$ we define
\begin{equation*}
Z(F)=
  \left\{\begin{aligned}
  &\emptyset &&\mbox{if }F=\emptyset\\
&F^{(\gamma)}  &&\mbox{if }F\neq \emptyset\mbox{ and } ht(F)=\gamma+1.
\end{aligned}
 \right.
\end{equation*}
Note that $Z(F)$ is a finite set.
If $\mathscr{W}$ is a family of compact scattered sets, then we set
$$Z(\mathscr{W})=\bigcup\left\{Z\left(\bigcap\mathscr{F}\right): \mathscr{F}\in [\mathscr{W}]^{<\omega}\right\}.$$
Define a sequence $\{\mathscr{U}_\beta:\beta<\kappa\}$ of subfamilies of $\mathscr{U}$ as follows:
\begin{equation*}
 \mathscr{U}_\beta=\left\{\begin{aligned}
                & \{U_\alpha:\alpha\leq\beta\}\cup \{U\in \mathscr{U}:\overline{U}\cap
                Z(\mathscr{U}'_\eta)\neq\emptyset\} && \mbox{ if } \beta=\eta+1\\
                & \bigcup\{\mathscr{U}_\alpha:\alpha<\beta\} && \mbox{ if } \beta \mbox{ is a limit ordinal.}
               \end{aligned}
\right.
\end{equation*}

\begin{claimnumber}\label{Claim1}
For each $\beta<\kappa$ we have $\mathscr{U}_\beta\subseteq \mathscr{U}_{\beta+1}$ and thus the sequence $\{\mathscr{U}_\beta:\beta<\kappa\}$ is increasing.
\end{claimnumber}
\begin{proof}
First note that if $\mathscr{U}_\alpha\subseteq\mathscr{U}_\gamma$, then
$\mathscr{U}'_\alpha\subseteq\mathscr{U}'_\gamma$ whence $Z(\mathscr{U}'_\alpha)\subseteq Z(\mathscr{U}'_\gamma)$.
We prove the claim inductively with respect to $\beta<\kappa$. Fix $\beta<\kappa$ and suppose that $\mathscr{U}_\alpha\subseteq \mathscr{U_\gamma}$ for $\alpha<\gamma\leq \beta$. We need to show that $\mathscr{U}_\beta\subseteq \mathscr{U}_{\beta+1}$. Let us consider the following two cases:

\smallskip

\textit{Case 1:} $\beta=\eta+1$ for some ordinal $\eta$. In this case
\begin{align*}
\mathscr{U}_{\beta+1}&=\{U_\alpha:\alpha\leq\beta+1\}\cup\{U\in \mathscr{U}:\overline{U}\cap Z(\mathscr{U}'_{\eta+1})\neq \emptyset\} \mbox{ and}\\
\mathscr{U}_{\beta}&=\{U_\alpha:\alpha\leq\beta\}\cup\{U\in \mathscr{U}:\overline{U}\cap Z(\mathscr{U}'_{\eta})\neq \emptyset\}
\end{align*}
Since $\mathscr{U}'_{\eta}\subseteq \mathscr{U}'_{\eta+1}$ by the inductive assumption, we are done by the observation made at the beginning of the proof.

\smallskip

\textit{Case 2:} $\beta$ is a limit ordinal. In this case
\begin{align*}
\mathscr{U}_{\beta+1}&=\{U_\alpha:\alpha\leq\beta+1\}\cup\{U\in \mathscr{U}:\overline{U}\cap Z(\mathscr{U}'_{\beta})\neq \emptyset\} \mbox{ and}\\
\mathscr{U}_{\beta}&=\bigcup\{\mathscr{U}_\alpha:\alpha<\beta\}.
\end{align*}
Pick $U\in \mathscr{U}_\beta$ and let $\xi<\beta$ be the first ordinal for which $U\in \mathscr{U}_\xi$. It is easy to see that $\xi$ must me a successor ordinal and thus either
$U=U_\alpha$ for some $\alpha\leq\xi<\beta$, in which case we are done, or $\overline{U}\cap Z(\mathscr{U'_\zeta})\neq\emptyset$, where $\xi=\zeta+1$. But $Z(\mathscr{U}'_\zeta)\subseteq Z(\mathscr{U}'_\beta)$, by the inductive assumption, so $U\in \mathscr{U}_{\beta+1}$.
\end{proof}

It is easy to see that
\begin{equation}\label{union is U}
 \mathscr{U}=\bigcup\{\mathscr{U}_\beta:\beta<\kappa\}.
\end{equation}

Let us show the following:

\begin{claimnumber}\label{Claim2}
If $\beta<\kappa$, then $|\mathscr{U}_\beta|\leq |\beta|\cdot\omega<\kappa.$
\end{claimnumber}
\begin{proof}
Note that $|\mathscr{U}'_\eta|\leq |\mathscr{U}_\eta|\cdot\omega$ for every $\eta<\kappa$. Thus, we can prove Claim \ref{Claim2} inductively with respect to
$\beta<\kappa$ in a similar way as in the proof of \cite[Proposition]{A} (cf. \cite[Lemma 3.5, Claim]{AK}). Let us enclose the argument for the convenience of the reader.

If $\beta$ is a limit ordinal then $\mathscr{U}_\beta=\bigcup\{\mathscr{U}_\alpha:\alpha<\beta\}$ and we are done. Suppose that $\beta=\eta+1$ for some ordinal $\eta$ and suppose that $|\mathscr{U}_\alpha|\leq|\alpha|\cdot\omega$ for $\alpha\leq \eta$. It is sufficient to check that
\begin{equation}\label{eq_claim}
 |\{U\in \mathscr{U}:\overline{U}\cap Z(\mathscr{U}'_\eta)\neq\emptyset\}|\leq |\eta|\cdot\omega.
\end{equation}
Since $|\mathscr{U}'_\eta|\leq |\mathscr{U}_\eta|\cdot\omega\leq |\eta|\cdot\omega$, we get
$$|Z(\mathscr{U}'_\eta)|=\Big| \bigcup\left\{Z\left(\bigcap\mathscr{F}\right): \mathscr{F}\in [\mathscr{U'_\eta}]^{<\omega}\right\}\Big|\leq |\eta|\cdot\omega.$$
Note that $Z(\mathscr{U}'_\eta)\subseteq X$, by \eqref{equation_sigma_compact} and the definition of $Z(\mathscr{U}'_\eta)$. So
\eqref{eq_claim} follows from our assumption that for every $x\in X$ the family $\{U\in \mathscr{U}:x\in \overline{U}\}$ is countable.
\end{proof}

For each $\beta<\kappa$ define

$$\mathscr{W}_\beta=\mathscr{U}_{\beta+1}\setminus \mathscr{U}_\beta.$$
Using \eqref{union is U} and definitions of the families $\mathscr{U}_\beta$ and $\mathscr{W}_\beta$, it is easy to see that
$$\mathscr{U}=\bigcup\{\mathscr{W}_\beta:\beta<\kappa\}.$$ Moreover, by Claim \ref{Claim2}, $|\mathscr{W}_\beta|<\kappa$ for all $\beta<\kappa$.
Applying the inductive assumption to $\mathscr{W}_\beta$ we find a family $\mathscr{V}_\beta$ of open subsets of $Z$ such that:
\begin{itemize}
 \item $\mathscr{V}_\beta$ refines $\mathscr{W}_\beta$ and $\bigcup\mathscr{V}_\beta=\bigcup\mathscr{W}_\beta$,
 \item $\overline{V}\cap X$ is $\sigma$-compact for every $V\in \mathscr{V}_\beta$,
 \item $\mathscr{V}_\beta$ is block-point-finite in $X$.
\end{itemize}
Let
$$\mathscr{V}=\bigcup\{\mathscr{V}_\beta:\beta<\kappa\}.$$
Clearly $\overline{V}\cap X$ is $\sigma$-compact for every $V\in \mathscr{V}$.
It is also easy to see that $\mathscr{V}$ refines $\mathscr{U}$ and $\bigcup\mathscr{U}=\bigcup\mathscr{V}$. It remains to show that $\mathscr{V}$ is block-point-finite in $X$. Each family $\mathscr{V}_\beta$ is block-point-finite in $X$, so for each $\beta<\kappa$, we can find a set $\Gamma_\beta$ such that $\mathscr{V}_\beta=\bigcup\{\mathscr{V}^\gamma_\beta:\gamma\in \Gamma_\beta\}$, where each family $\mathscr{V}^\gamma_\beta$ is countable and the family
$$\mathscr{R}_\beta=\left\{\bigcup\mathscr{V}^\gamma_\beta:\gamma\in \Gamma_\beta\right\}$$
is point-finite in $X$.

We may assume that the sets $\Gamma_\beta$ are pairwise disjoint, for $\beta<\kappa$.
Let $\Gamma=\bigcup_{\beta<\kappa}\Gamma_\beta$. For $\gamma\in \Gamma$ define
$$\mathscr{V}^\gamma=\mathscr{V}^\gamma_\beta,$$
where $\beta<\kappa$ is the unique ordinal such that $\gamma\in \Gamma_\beta$.
We have $\mathscr{V}=\bigcup\{\mathscr{V}^\gamma:\gamma\in \Gamma\}$ and each family $\mathscr{V}^\gamma$ is countable.

We will show that the indexed family
$$\mathscr{R}=\left\{ \bigcup \mathscr{V}^\gamma:\gamma\in \Gamma \right\}$$
is point-finite in $X$ (and hence $\mathscr{V}$ is block-point-finite in $X$).

Striving for a contradiction, suppose that there is $x\in X$ for which the set
$\{\gamma\in \Gamma:x\in \bigcup\mathscr{V}^\gamma\}$
is infinite.
Since $\mathscr{R}=\bigcup\{\mathscr{R}_\beta:\beta<\kappa\}$ and each family $\mathscr{R}_\beta$ is point-finite, there
is an increasing sequence of ordinals $\beta_1<\beta_2<\ldots$ and a sequence of indices $\gamma_1,\gamma_2,\ldots$ such that
$$\gamma_k\in \Gamma_{\beta_k}\quad  \mbox{and}\quad x\in \bigcup \mathscr{V}^{\gamma_k}_{\beta_k} \quad\mbox{for all } k.$$
For each $k$, find $V_k\in \mathscr{V}^{\gamma_k}_{\beta_k}$ such that $x\in V_k$. Since $\mathscr{V}_{\beta_k}$ refines $\mathscr{W}_{\beta_k}$, we can find $W_k\in \mathscr{W}_{\beta_k}=\mathscr{U}_{(\beta_k+1)}\setminus \mathscr{U}_{\beta_{k}}$ with $x\in W_k$. Since $W_k\in \mathscr{U}_{(\beta_k+1)}\subseteq \mathscr{U}$ and $x\in X$, according to \eqref{equation_sigma_compact}, for each $k\geq 1$, there is $n_k\in\omega$ such that $x\in K_{n_k}(W_k)$. For $m=1,2,\ldots$, define
$$C_m=K_{n_1}(W_1)\cap\dotsb\cap K_{n_m}(W_m).$$
The sets $C_m$ are nonempty compact and scattered, so for each $m$, there is an ordinal $\zeta_m$ such that $ht(C_m)=\zeta_m+1$. The sequence $C_1\supseteq C_2\supseteq\ldots$ is decreasing, so $\zeta_{i+1}\leq \zeta_i$. Hence, there must be $m$ such that $\zeta_{i+1}=\zeta_i$ for $i\geq m$. From the choice of $m$ and using $\overline{W_{m+2}}\supseteq K_{n_{m+2}}(W_{m+2})$, we get
$$\overline{W_{m+2}}\cap Z(C_m)\supseteq Z(C_{m+2}).$$
Since the latter set is nonempty, we have
$$\overline{W_{m+2}}\cap Z(K_{n_1}(W_1)\cap\dotsb\cap K_{n_m}(W_m))\neq\emptyset.$$
Since $K_{n_k}(W_k)\in \mathscr{U}'_{(\beta_k+1)}$ and $\beta_1<\beta_2<\ldots$, we infer from Claim \ref{Claim1} that
$$K_{n_k}(W_k)\in \mathscr{U}'_{(\beta_m+1)} \mbox{ for all } k\leq m.$$
It follows that $W_{m+2}\in \mathscr{U}_{(\beta_m+2)}$. On the other hand $W_{m+2}\in \mathscr{W}_{\beta_{m+2}}$ so
$$W_{m+2}\notin \mathscr{U}_{\beta_{m+2}}\supseteq \mathscr{U}_{(\beta_{m+1}+1)}\supseteq \mathscr{U}_{(\beta_{m}+2)}$$
(cf. Claim \ref{Claim1}),
which is a contradiction.
\end{proof}

We are now ready to give a proof of our main result. The reasoning presented below is modelled on the proof of Theorem 3.2 in \cite{Tk}.
Let us fix some notation.
Given a point $z$ in a topological space $T$, by $\chi(z,T)$ we denote the minimal size of a local base at $z$ in the space $T$. A space $T$ is \textit{first-countable} if $\chi(z,T)\leq\omega$ for all $z\in T$.

\begin{proof}[Proof of Theorem \ref{main}]
Let $\{Y_\gamma:\gamma\in \Gamma\}$ be a family of first-countable spaces and let $a=(a_\gamma)_{\gamma\in \Gamma}\in \prod_{\gamma\in \Gamma}Y_\gamma$.
We proceed by induction on the Cantor-Bendixson rank $ht(X)$ of $X$. Suppose that $ht(X)=\alpha$ and the result holds true for spaces of rank $\beta<\alpha$. If $\alpha$ is a limit ordinal then we are easily done. Indeed, in this case each point $x\in X$ has a neighborhood $U_x$ with $ht(\overline{U_x})<\alpha$. By the inductive assumption, $\overline{U_x}$ is $\sigma$-compact being Lindel\"of. Since $X$ is Lindel\"of, the open cover $\{U_x:x\in X\}$ has a countable subcover $\{U_{x_i}:i=1,2,\ldots\}$. It follows that $X=\bigcup\{\overline{U_{x_i}}:i=1,2,\ldots\}$ is $\sigma$-compact being a countable union of $\sigma$-compact sets.

 Suppose that $\alpha=\gamma+1$ for some ordinal $\gamma$ and that the result is true for spaces of rank $\beta\leq \gamma$. The set $X^{(\gamma)}$ is a closed and discrete subset of $X$, so by the Lindel\"of property of $X$, it is countable. Therefore we can assume, without loss of generality that
$X^{(\gamma)}=\{p\}$ is a singleton. Since $p\in \Sigma(\prod_{\gamma\in \Gamma}Y_\gamma,a)$, we have $\Sigma(\prod_{\gamma\in \Gamma}Y_\gamma,a)=\Sigma(\prod_{\gamma\in \Gamma}Y_\gamma,p)$. Thus, we can also assume that $p=a$ (see \cite[p.175]{Tk}).
For each $\gamma\in \Gamma$ let $K_\gamma$ be a compactification of $Y_\gamma$. We have $\chi(a_\gamma,K_\gamma)=\chi(a_\gamma,Y_\gamma)\leq \omega$ (cf. \cite[p.175]{Tk}).
Hence, for each $\gamma\in \Gamma$, we can find a countable family $\{W_{\gamma,n}:n\in \omega\}$ of open subsets of $K_\gamma$ such that
\begin{equation}\label{equation_intersection}
\overline{W_{\gamma,n+1}}\subseteq W_{\gamma,n} \mbox{ and }\bigcap_{n\in \omega}W_{\gamma,n}=\bigcap_{n\in \omega}\overline{W_{\gamma,n}}=\{a_\gamma\}.
\end{equation}
For $n\in\omega$ and $\xi\in \Gamma$ we set
$$U_{\xi,n}=\{z\in \prod_{\gamma\in\Gamma}K_\gamma:z(\xi)\notin \overline{W_{\xi, n}}\}.$$
Let
$$\mathscr{U}=\{U_{\xi,n}:\xi\in \Gamma,\;n\in \omega\}$$
and note that $\mathscr{U}$ consists of open subsets of $\prod_{\gamma\in \Gamma}K_\gamma$.
We infer from \eqref{equation_intersection} that
if $x\in \Sigma(\prod_{\gamma\in \Gamma}K_\gamma,a)$, in particular if $x\in X$, then the family $\{U\in\mathscr{U}:x\in \overline{U}\}$ is countable.
 Moreover,
 \begin{equation}\label{union}
  \bigcup \mathscr{U}=\prod_{\gamma\in \Gamma}K_\gamma\setminus \{a\}.
 \end{equation}
 If $U\in \mathscr{U}$, then $\overline{U}\cap X$ is Lindel\"of scattered being a closed subspace of a Lindel\"of scattered space $X$.
 Since $a\notin \overline{U}$, we have $ht(\overline{U}\cap X)<\alpha$ whence $\overline{U}\cap X$ is $\sigma$-compact by the inductive assumption.
 We can therefore apply Lemma \ref{Alster} to the family $\mathscr{U}$. Consequently, we get a family $\mathscr{V}$ of open subsets of $\prod_{\gamma\in \Gamma}K_\gamma$ that refines $\mathscr{U}$ and has the following properties:
 \begin{enumerate}[(a)]
  \item $\bigcup\mathscr{V}=\bigcup\mathscr{U}$
  \item $\overline{V}\cap X$ is $\sigma$-compact and
  \item $\mathscr{V}$ is block-point-finite in $X$.
\end{enumerate}
By (c), we can write $\mathscr{V}=\bigcup\{\mathscr{V}_t:t\in T\}$ where each family $\mathscr{V}_t$ is countable and the family $\{\bigcup \mathscr{V}_t:t\in T\}$ is point-finite in $X$. For $t\in T$, let
$$V_t=\bigcup\mathscr{V}_t\quad \mbox{and}\quad H_t=\bigcup\left\{\overline{V}:V\in \mathscr{V}_t\right\}.$$
Since the family $\mathscr{V}_t$ is countable, we infer from (b) that the set $H_t\cap X$ is $\sigma$-compact. Hence, the set
$$G_t=V_t\setminus X=V_t\cap \left(\prod_{\gamma\in \Gamma}K_\gamma\setminus (H_t\cap X)\right)$$
is a $G_\delta$-subset of $\prod_{\gamma\in \Gamma}K_\gamma$ which is disjoint from $X$, for all $t\in T$. The family $\{V_t:t\in T\}$ is an open expansion of the family $\{G_t:t\in T\}$ which is point-finite in $X$. According to Lemma \ref{lemma_Tkachuk}, there is a $G_\delta$-subset $G$ of $\prod_{\gamma\in \Gamma}K_\gamma$ such that
$$\bigcup\{G_t:t\in T\}\subseteq G \subseteq \prod_{\gamma\in \Gamma}K_\gamma\setminus X.$$
It follows from $\eqref{union}$ and the property (a) that $\bigcup\{V_t:t\in T\}=\prod_{\gamma\in \Gamma}K_\gamma\setminus \{a\}$. Since $a\in X$, we get
$$\bigcup\{G_t:t\in T\}= \left(\bigcup\{V_t:t\in T\}\right)\setminus X =\prod_{\gamma\in \Gamma}K_\gamma\setminus X.$$
Hence, $G=\prod_{\gamma\in \Gamma}K_\gamma\setminus X$ is a $G_\delta$-subset of the compact space $\prod_{\gamma\in \Gamma}K_\gamma$ whence $X$ is $\sigma$-compact.
\end{proof}
Theorem \ref{main} implies the following result which, in particular, answers a question of Leiderman and Tkachuk from \cite{LT} (see \cite[Question 4.5]{LT}). We refer the reader to \cite{KL} and \cite{LT} for a definition and basic facts about $\Delta$-spaces.
\begin{cor}\label{corollary-delta-spaces}
If $X$ is a Lindel\"of scattered subspace of a $\Sigma$-product $\Sigma(\mathbb{R}^\Gamma)$, then $X$ is a $\Delta$-space.
\end{cor}
\begin{proof}
Apply Theorem \ref{main}, \cite[Theorem 3.7]{KL1} and \cite[Proposition 2.2]{KL}.
\end{proof}

\section*{Acknowledgements}
The authors were partially supported by
Fundaci\'{o}n S\'{e}neca - ACyT Regi\'{o}n de Murcia project 21955/PI/22, Agencia Estatal de Investigación (Government of Spain) and Project PID2021-122126NB-C32 funded by
MICIU/AEI /10.13039/\\
501100011033/ and FEDER A way of making Europe (A. Avil\'es and M. Krupski); European Union - NextGenerationEU funds through Mar\'{i}a Zambrano fellowship and
the NCN (National Science Centre, Poland) research Grant no.\\
2020/37/B/ST1/02613 (M. Krupski)

\bibliographystyle{siam}
\bibliography{bib.bib}
\end{document}